\documentclass[12pt]{amsart}
 
\usepackage{amsmath, amssymb, amscd, eucal}

\usepackage{latexsym}

\usepackage{graphicx}

\usepackage{amsfonts, amsthm, indentfirst}

\usepackage{hyperref}

\usepackage{graphicx}

\usepackage{setspace}

\usepackage{color}

\numberwithin{equation}{section}

\newtheorem{thm}{Theorem}[section]

\newtheorem{prop}{Proposition}[section]
\newtheorem{cor}{Corollary}[section]

\theoremstyle{definition}

\theoremstyle{remark}
\newtheorem{rem}{Remark}[section]

\newcommand{\Ric}{\mbox{Ric}}
\newcommand{\R}{\mathbb R}

\newcommand{\be}{\begin{equation}}
\newcommand{\ee}{\end{equation}}
\newcommand{\bee}{\begin{equation*}}
\newcommand{\eee}{\end{equation*}}
\newcommand{\bal}{\begin{aligned}}
\newcommand{\eal}{\end{aligned}}

\def\p{\partial}

\def\Pi{\displaystyle{\mathbb{II}}}

\def\m{\mathfrak{m}}

\def\c{\mathfrak{c}}
\def\C{\mathcal{C}}

\begin{document}

\title{Implications of some mass-capacity inequalities}

\author{Pengzi Miao}
\address[Pengzi Miao]{Department of Mathematics, University of Miami, Coral Gables, FL 33146, USA}
\email{pengzim@math.miami.edu}

\begin{abstract}
Applying a family of mass-capacity related inequalities proved in \cite{M22},  
we obtain sufficient conditions that imply 
the nonnegativity as well as positive lower bounds of the mass, 
on a class of manifolds with nonnegative scalar curvature, with or without a singularity.  
\end{abstract}

\maketitle

\markboth{Pengzi Miao}{Implications of some mass-capacity inequalities}

\section{Introduction} 

A smooth Riemannian $3$-manifold $(M, g)$ is called asymptotically flat (AF)  if $M$,  
outside a compact set, is diffeomorphic to $ \R^3$ minus a ball;
the associated metric coefficients  satisfy   
$$
g_{ij} = \delta_{ij} + O ( |x |^{-\tau} ), \ 
 \p g_{ij} = O ( |x|^{-\tau -1}) ,  \ \ \p \p g_{ij} = O (|x|^{-\tau -2} ), 
$$
for some $ \tau > \frac12$; and the scalar curvature  of $g$ is  integrable.
Under these AF conditions, the limit, near $\infty$, 
$$
\m  = \lim_{ r \to \infty  } \frac{1}{16 \pi} \int_{ |x | = r } \sum_{j, k} ( g_{jk,j} -  g_{jj,k} ) \frac{x^k}{ |x| } 
$$
exists and is called the ADM mass  \cite{ADM61} of $(M,g)$. 
It is a result of Bartnik \cite{Bartnik86}, and of Chru\'{s}ciel \cite{Chrusciel86}, 
that $\m$ is a geometric invariant, independent on the choice of the coordinates $\{ x_i \}$.

A fundamental result on the mass and the scalar curvature is the Riemannian positive mass theorem (PMT): 

\begin{thm}[\cite{SchoenYau79, Witten81}] \label{thm-PMT}
Let $(M, g)$ be a complete, asymptotically flat $3$-manifold with nonnegative scalar curvature without boundary. Then
$$ \m \ge 0  , $$
and equality  holds if and only if $(M, g)$ is isometric to the Euclidean space $ \R^3$. 
\end{thm}

On an asymptotically flat $3$-manifold $(M,g)$ with boundary $ \Sigma = \p M $, 
the capacity (or $L^2$-capacity) of $\Sigma $ is defined by 
$$
\c_{_\Sigma} = \inf_f \,  \left\{  \frac{1}{4 \pi} \int_M | \nabla f |^2  \right\}, 
$$
where the infimum is taken over all locally Lipschitz functions $f$ that vanishes on $\Sigma$ and tend to $1$ at infinity. 
Equivalently, if $\phi$ denotes the function with 
$$
\Delta \phi = 0 , \ \ 
 \phi |_\Sigma = 1 , \ \text{and} \  
 \phi \to 0  \ \ \text{at} \  \ \infty , 
$$
then, 
$ \displaystyle  \c_{_\Sigma} = \frac{1}{4 \pi} \int_M | \nabla \phi |^2  = \frac{1}{4\pi} \int_{\Sigma} | \nabla \phi| $, 
and
$$ \phi = \c_{_\Sigma} |x|^{-1} + o ( | x|^{-1} ) , \ \  \text{as} \ x \to \infty . $$ 

Regarding the mass and the capacity, if $\Sigma$ is a minimal surface, 
Bray showed

\begin{thm}[\cite{Bray02}] \label{thm-Bray}
Let $(M, g)$ be a complete, asymptotically flat $3$-manifold with nonnegative scalar curvature, with minimal surface boundary $\Sigma = \p M$. Then
$$ \m \ge \c_{_\Sigma}  ,$$
and equality   holds iff $(M, g)$ is isometric to a spatial Schwarzschild manifold 
outside the horizon. 
\end{thm}

In \cite[Theorem 7.4]{M22}, an inequality relating 
the mass-to-capacity ratio to the Willmore functional of the boundary
was obtained: 

\begin{thm} [\cite{M22}] \label{thm-intro-1-1}
 Let $(M, g)$ be a complete, orientable, asymptotically flat $3$-manifold with one end, with boundary $\Sigma$.
Suppose $ \Sigma$ is connected and $H_2 (M, \Sigma) = 0$.
If $ g$ has nonnegative scalar curvature, then
\be \label{eq-main-m-c-ratio}
 \frac{ \m} {  \c_{_\Sigma} }   \ge 1  -  \left( \frac{1}{16\pi}  \int_\Sigma   H^2 \right)^\frac12  .
\ee
Here $\m$ is the mass of $(M,g)$, $\c_{_\Sigma}$ is the capacity of $\Sigma$ in $(M,g)$,  
and $H$ is the mean curvature of $\Sigma$. 
Moreover, equality in \eqref{eq-main-m-c-ratio} holds if and only if 
$(M, g)$ is isometric to a spatial Schwarzschild manifold outside a rotationally symmetric sphere with
nonnegative mean curvature.
\end{thm}

As shown in \cite{M22}, \eqref{eq-main-m-c-ratio} implies the 
$3$-dimensional PMT.  
For instance, assuming $M$ is topologically $ \R^3$, applying \eqref{eq-main-m-c-ratio} to 
the exterior of a geodesic sphere $S_r$  with radius $r $ centered at any point $p \in M$, one has 
\be \label{eq-basic-app}
 \frac{ \m} {  \c_{_{S_r} } }   \ge  1  -  \left( \frac{1}{16\pi}  \int_{S_r}   H^2 \right)^\frac12  .
\ee
Letting $ r \to 0$,  one obtains $ \m \ge 0 $. 
Earlier proofs of $3$-d PMT via harmonic functions 
were given by Bray-Kazaras-Khuri-Stern  \cite{BKKS}
and Agostiniani-Mazzieri-Oronzio \cite{AMO21}.

\vspace{.2cm}

Theorem \ref{thm-intro-1-1} follows from two other results (Corollary 7.1 and Theorem 7.3) in \cite{M22}:  

\begin{thm} [\cite{M22}]
Let $(M, g)$ be a complete, orientable, asymptotically flat $3$-manifold with one end, with connected 
boundary $\Sigma$, satisfying  $H_2 (M, \Sigma) = 0$. If $ g$ has nonnegative scalar curvature, then 
\be \label{eq-gdu-H} 
 \left( \frac{1}{\pi} \int_\Sigma  |  \nabla u |^2 \right)^\frac12  \le  \left( \frac{1}{16\pi}  \int_\Sigma   H^2 \right)^\frac12  + 1 ,
\ee
and
\be \label{eq-gdu-mc}
\frac{\m}{ 2  \c_{_\Sigma} } \ge 1 - 
\left( \frac{1}{4 \pi} \int_{\Sigma} | \nabla u |^2 \right)^\frac12 .
\ee
Here $ u$ is the harmonic function with $ u = 0 $ at $ \Sigma$ and $ u \to 1 $ near $\infty$.
Moreover,  
\begin{itemize}
\item  equality in \eqref{eq-gdu-H} holds if and only if 
 $(M, g)$ is isometric to a spatial Schwarzschild manifold  outside a rotationally symmetric sphere
 with nonnegative mean curvature; 
 \item  equality in \eqref{eq-gdu-mc} holds if and only if 
 $(M, g)$ is isometric to a spatial Schwarzschild manifold  outside a rotationally symmetric sphere. 
\end{itemize} 
 \end{thm}

A corollary of \eqref{eq-gdu-H} (see \cite[Theorem 7.2]{M22}) is 
an upper bound on the capacity-to-area-radius ratio,
first derived by Bray and the author \cite{BM08}.

\begin{thm} [\cite{BM08}] \label{thm-BM}
Let $(M, g)$ be a complete, orientable, asymptotically flat $3$-manifold with one end, with boundary $\Sigma$.
Suppose $ \Sigma$ is connected and $H_2 (M, \Sigma) = 0$. 
If $g$ has nonnegative scalar curvature, then
\be \label{eq-Bray-Miao-06} 
 \frac{  2  \c_{_\Sigma} } {  r_{_\Sigma}  }  \le  \left( \frac{1}{16\pi}  \int_\Sigma   H^2 \right)^\frac12  + 1 .
\ee
Here $\c_{_\Sigma}$ is the capacity of $\Sigma$ in $(M, g)$ and 
$ r_{_\Sigma} = \left( \frac{ | \Sigma | }{4 \pi}  \right)^\frac12 $ is the area-radius of $\Sigma$.
Moreover, equality holds if and only if 
$(M, g)$ is isometric to a spatial Schwarzschild manifold outside a rotationally symmetric sphere 
with nonnegative mean curvature. 
\end{thm}

In this paper, we give some other applications of \eqref{eq-main-m-c-ratio}, \eqref{eq-gdu-H} and \eqref{eq-gdu-mc}.

\vspace{.2cm}

First, for later purposes, we remark on the topological assumption ``$H_2 (M, \Sigma) = 0 $" in
Theorems \ref{thm-intro-1-1} -- \ref{thm-BM} above:  
the assumption is imposed only to ensure each regular level set of the harmonic function $u$, vanishing at the boundary
and tending to $1$ near $\infty$, to be connected in the interior of $M$ (see the paragraph preceding the 
proof of Theorem 3.1 in \cite{M22}); 
indeed, 
\eqref{eq-main-m-c-ratio}, \eqref{eq-gdu-H} and \eqref{eq-gdu-mc} (and all other results from \cite{M22}) hold 
if ``$H_2 (M, \Sigma) = 0 $" is replaced by assuming   

\begin{itemize}
\item[($\ast$)] each closed, connected, orientable surface in the interior of $M$ either is the boundary of a bounded domain, or 
 together with $\Sigma$ forms the boundary of a bounded domain. 
\end{itemize}

Now we motivate the main tasks in this paper. 
Let us first return  to the setting of \eqref{eq-basic-app}, 
in which the surface  $S_r$ ``closes up nicely" (to bound a geodesic ball).
In this setting,  by a result of Mondino and Templeton-Browne \cite{MT22}, 
$\{ S_r \}$ can be perturbed to yield another family of surfaces $\{ \Sigma_r \}$ so that, as $ r \to 0$,  
\be
 \int_{\Sigma_r} H^2 
 =  16 \pi - \frac{ 8 \pi}{3} R(p) r^2 + \frac{4\pi}{3}\left[ \frac19 R(p)^2 - \frac{4 }{15} | \mathring{\Ric} (p) |^2  - \frac{1}{5} \Delta R (p) \right] r^4 + O (r^5) .
\ee
Here $ R$ denotes the scalar curvature and $ \mathring{\Ric} = \Ric - \frac13 R g $ is the traceless part of $ \Ric$, the Ricci tensor. 
Applying \eqref{eq-main-m-c-ratio} to the exterior of these $ \Sigma_r $ in $(M,g)$, one obtains 
\be \label{eq-m-c-ratio-small-p-ball}
 \frac{ \m} {  \c_{_{\Sigma_r} } }   \ge   
  \frac{1}{12} R(p) r^2  + \left[  \frac{1}{90} | \mathring{ \Ric } (p) |^2 - \frac{1}{864} R(p)^2   + \frac{1}{120} \Delta R (p) \right]  r^4 + O (r^5) .
\ee
If $ R \ge 0 $, \eqref{eq-m-c-ratio-small-p-ball} shows the inequality $ \m \ge 0 $ as well as the rigidity of $\m = 0 $. 

\vspace{.2cm}

In general,  \eqref{eq-main-m-c-ratio} suggests that, 
if it is applied to obtain $\m \ge 0 $ on an $(M,g)$, 
the manifold boundary $\Sigma$ does not need to admit a ``nice fill-in". 
Rewriting \eqref{eq-main-m-c-ratio} as
\bee
 \m \ge  \c_{_\Sigma} \left[ 1  -  \left( \frac{1}{16\pi}  \int_{S_r}   H^2 \right)^\frac12  \right] ,
\eee
one may seek conditions on metrics $g$ with a ``singularity" 
so that $\m \ge 0$ while $g$ is allowed to be incomplete.

Similarly, on an $(M,g)$ with two ends, one of which is asymptotically flat (AF), 
assuming it admits a harmonic function $u$ that tends to $1$ at the AF end and tends to $0$ at the other 
end, one may aim to apply \eqref{eq-gdu-mc}, i.e.  
\bee 
\m \ge  2  \c_{_\Sigma} \left[ 1 - \left( \frac{1}{4 \pi} \int_{\Sigma} | \nabla u |^2 \right)^\frac12 \right] , 
\eee
to bound $\m$ via the energy of $u$ on the entire $(M,g)$. 

Below we formulate a class of manifolds to carry out the above mentioned tasks. 
Throughout the paper, let $N$ be a noncompact,  connected, orientable $3$-manifold. 
We assume $N$ admits an increasing exhaustion sequence of bounded domains with connected boundary.
Precisely, this means there exists a sequence of closed, orientable surfaces $\{ \Sigma_k \}_{k=1}^\infty$
in $N $ such that 
\begin{itemize}
\item $ \Sigma_k$ is connected; 
\item $ \Sigma_k = \p D_k $ for a precompact domain $D_k \subset N$; and 
\item $ \bar D_k  \subset D_{k+1}  $  and $ N  = \cup_{k=1}^\infty D_k$. Here $\bar D_k = D_k \cup \Sigma_k $ 
is the closure of $D_k$ in $N$. 
\end{itemize}

Fix a point $ p \in N$,  let $ M = N \setminus \{p \}$. 
On $M$, let $g$ be a smooth metric that is asymptotically flat near $p$.  
We refer $p$ as the asymptotically flat (AF) $\infty$ of $(M,g)$. 
Unless otherwise specified, we do not impose assumptions on the behavior of $g$ near $ \Sigma_k $
as $ k \to \infty$. 
In particular, $(M,g)$ does not need to be complete,

Given any closed, connected surface $S \subset M$,
we say $S$ encloses $p$ if    $S = \p D_{_S} $ 
for some precompact domain $D_{_S} \subset N$
such that $ p \in D_{_S}$. 
Let $\mathcal{S}$ denote  the set of all such surfaces $S \subset M$ enclosing $p$. 
Clearly, $ \Sigma_k \in \mathcal{S}$ for large $k$.
Define
\be \label{eq-inf-c} 
\c (M,g) = \inf_{ S \in \mathcal{S}} \, \c_{_S} . 
\ee
Here $ \c_{_S}$ is the capacity of $ S $ in the asymptotically flat $(E_{_S}, g)$, 
where 
$$E_{_S} =  (  D_{_S} \setminus \{p\} ) \cup S . $$ 

As a functional on $\mathcal{S}$, 
the capacity $\c_{_S}$ has a monotone property, that is if $S_1, S_2 \in \mathcal{S}$ and $D_{_{S_1} } \subset D_{_{S_2}}$,
then $ \c_{_{S_1}} \ge \c_{_{S_2}}$. 
Such a property readily implies $\{ \c_{_{\Sigma_k}} \}$ is monotone non-increasing and 
\be 
\c (M,g) = \lim_{ k \to \infty}  \c_{_{\Sigma_k}} . 
\ee
Standard arguments show $\c (M, g) > 0 $ if and only if there exists 
a harmonic function $ w $ on $(M,g)$ such that $ 0 < w < 1 $ on $M$ and $ w (x) \to 1$ at $\infty$ (i.e. as $x \to p$).  
(See Proposition \ref{prop-c-m-positive} in Section \ref{sec-pmt-cm-positive}.) 

\begin{figure}[h]
\includegraphics[scale=0.32]{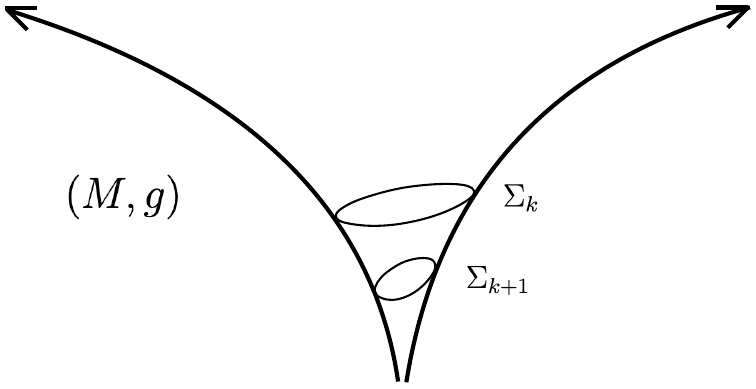}
 \ \ \hspace{2cm}
 \includegraphics[scale=0.3]{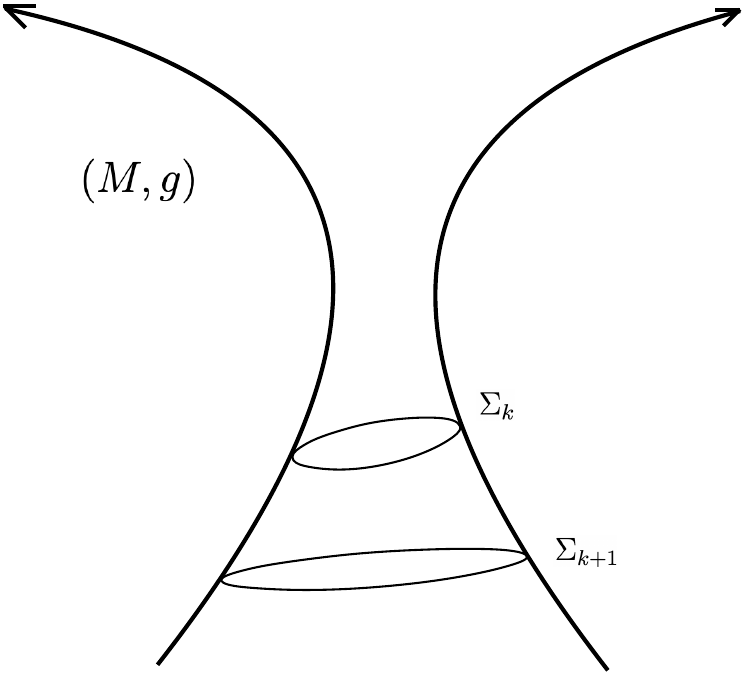}
\caption{On the left is an examples of $(M,g)$ with $\c (M,g) = 0$; the arrow denotes the AF end;
   $\{ \Sigma_k \}$ may approach a ``singularity" as $k \to \infty$.
   On the right is an example of $(M,g)$ with $\c (M,g) > 0 $; besides the AF end, $(M,g)$ 
  has another end with suitable growth.}
\end{figure}

\vspace{.2cm}

For manifolds $(M,g)$ with $\c(M,g) = 0 $, 
we seek conditions that imply the AF end of $(M,g)$ has mass $\m \ge 0 $, see Theorem \ref{thm-pmt-cw} 
and Remark \ref{rem-pmt-cone-horn}.
For $(M,g)$ with $ \c (M,g) > 0 $, we explore for sufficient conditions that bound  $\m $ from below via  $ \c (M,g)$, 
see Theorem \ref{thm-c-m-u-limit} and Corollary \ref{thm-m-c-Ric}.

\section{Singular metrics with $ \m \ge 0$}  \label{sec-pmt-cm-zero}

Let $N$, $M$ and $g$ be given in the definition of $\c (M,g)$ in \eqref{eq-inf-c}. 
Given $ S \in \mathcal{S}$,  let  
$$ W(S) = \int_S   H^2 . $$ 
We want to apply \eqref{eq-main-m-c-ratio} to $(E_{_S}, g)$. 
For this purpose, we assume the background manifold $N$ satisfies $H_2 (N) = 0$. 
Under this assumption,  any closed, connected surface $S' $ in $M = N \setminus \{ p \}$
is the boundary of a bounded domain $D \subset N$. If $ p \not \in D$, then $D \subset M$;
if $ p \in D$, then $ S'$ is homologous to $S \in \mathcal{S}$. Therefore, condition $(\ast)$ holds on $E_{_S}$.

The following is a direct corollary of \eqref{eq-main-m-c-ratio}.

\begin{prop} \label{prop-pmt-c-w}
Suppose $H_2 (N) = 0 $ and $(M, g)$ has nonnegative scalar curvature.  
Then
\be \label{eq-c-w-pmt}
 \c_{_{S_k}} W(S_k)^\frac12 \to 0  \ \text{along a sequence } \{ S_k \} \subset \mathcal{S} \ \ 
\Rightarrow \  \m \ge 0 . 
\ee
\end{prop}

\begin{proof}
If $ W(S_k) \le 16 \pi $ for some $k$, then \eqref{eq-main-m-c-ratio} implies $ \m \ge 0 $.

Suppose $  W(S_k) > 16 \pi $ for every $k$, then 
``$ \c_{_{S_k}} W(S_k)^\frac12 \to 0  $" implies 
``$  \c_{_{S_k}}  \to 0 $". 
Rewriting  \eqref{eq-main-m-c-ratio} as 
\be \label{eq-main-m-c-ratio-app}
\m \ge  \c_{_{S_k} } \left[ 1  -  \left( \frac{1}{16\pi}  W(S_k)  \right)^\frac12 \right] 
\ee
and letting $k \to \infty$, we have $\m \ge 0$.
\end{proof} 

Given $ S \in \mathcal{S}$, 
let  $ \displaystyle  \m_{_H} (S) = \frac{r_{_S} }{2} \left( 1 -  \frac{1}{16\pi} W (S) \right)  $ 
denote the Hawking mass of $S$ (\cite{Hawking68}).
Inequality \eqref{eq-m-c-m-mh} in the next Proposition 
is comparable to the result of 
Huisken and Ilmanen \cite{HI01} on the relation between $\m$ and $ \m_{_H} (S)$.

\begin{prop}
Suppose $H_2 (N) = 0 $ and $(M, g)$ has nonnegative scalar curvature.
If  a surface $S \in \mathcal{S} $ satisfies $ W(S) \ge 16 \pi $,  then
\be \label{eq-m-c-m-mh}
\m \ge 
 \c_{_S}  \left[ 1 -   \left( \frac{1}{16\pi}  W(S)  \right)^\frac12  \right]   \ge   \m_{_H} (S).
\ee
\end{prop}

\begin{proof}
If $ W(S) \ge 16 \pi $, then \eqref{eq-Bray-Miao-06} implies 
\be \label{eq-Bray-Miao-06-app} 
\begin{split}
 \c_{_S}  \left[ 1 -   \left( \frac{1}{16\pi} \int_S   H^2 \right)^\frac12  \right]   \ge & \
 \frac{r_{_S} }{2} \left[ 1 -  \frac{1}{16\pi}  \int_S   H^2  \right] 
 = \m_{_H} (S) .
\end{split}
\ee
This combined with \eqref{eq-main-m-c-ratio} proves  \eqref{eq-m-c-m-mh}.
\end{proof}

\begin{rem}
Similar to \eqref{eq-c-w-pmt}, a condition of ``$ r_{_{S_k}} W(S_k) \to 0$" along $\{ S_k \} \subset \mathcal{S} $ also implies ``$ \m \ge 0 $". 
However, if $ \inf_{k} W(S_k) \ge 16 \pi  $,  then  
$$ ``r_{_{S_k}} W(S_k) \to 0" \ \Rightarrow  \  ``r_{_{S_k}} \to 0 \ \text{and} \   r_{_{S_k}} W(S_k)^\frac12 \to 0" \ 
\Rightarrow  \ ``\c_{_{S_k}} \to 0" , $$
where the last step is by \eqref{eq-Bray-Miao-06}.
Combined with \eqref{eq-Bray-Miao-06-app}, this implies  
the assumption of ``$ \c_{_{S_k}} W(S_k)^\frac12 \to 0  $" in Proposition \ref{prop-pmt-c-w}. 
\end{rem}

\vspace{.1cm}

In what follows, let $ \{ \Sigma_k \} \subset \mathcal{S} $ be the sequence of surfaces given in the introduction.  
The numerical value of $ \c_{_{\Sigma_k}}$ depends on $g$ near the AF end. However, a property of 
``$ \c_{_{\Sigma_k}}  \to 0  $" does not. This was shown by Bray and Jauregui \cite{BJ13} 
in the context of $(M,g)$ having a zero area singularity. 
Their argument 
applies to ``$ \c_{_{\Sigma_k}} W(\Sigma_k)^\frac12 \to 0  $". To illustrate this, it is convenient to 
adopt a notion of relative capacity (see \cite{J21} for instance).  
Given two surfaces $S, \tilde S \in \mathcal{S}$, suppose $S \cap \tilde S = \emptyset$ and 
$ D_{_{\tilde S}} \subset D_{_S}$. The capacity of $S$ relative to $\tilde S$ is 
\be
\c_{_{(S, \tilde S)} }= \frac{1}{4\pi} \int_{D_{_{S}}  \setminus D_{_{\tilde S } } } | \nabla v |^2 ,
\ee
where $v$ is the harmonic function on $D_{_{S}}  \setminus D_{_{\tilde S } }$ with $ v = 0 $ at $ S$ and $ v = 1 $ at $ \tilde S$. 

\begin{prop} \label{prop-c-w-localized}
Let $\tilde S \in \mathcal{S}$ be a fixed surface. Then, as $k \to \infty$, 
$$ \c_{_{\Sigma_k}} W(\Sigma_k)^\frac12 \to 0  \ \Longleftrightarrow \  \c_{_{(\Sigma_k, \tilde S) }} W(\Sigma_k)^\frac12 \to 0  . $$
\end{prop}

\begin{proof}
For large $k$, let $u_k$, $v_k$ be the harmonic function on $D_{k} \setminus \{ p \}$,
$ D_{k} \setminus D_{_{\tilde S}}$, with boundary values $u_k = 0 $ at $\Sigma_k$, $u_k \to1$ at the AF $\infty$, 
$v_k = 0 $ at $\Sigma_k$, $v_k = 1 $ at $\tilde S$, respectively. 
Let $ \beta_k = \min_{\tilde S} u_k $. By the maximum principle, $ v_k \ge u_k \ge \beta_k v_k $ on 
$ D_{k} \setminus D_{_{\tilde S}}$, which implies 
$   \p_\nu v_k \ge  \p_\nu u_k \ge \beta_k \p_\nu v_k$ at $\Sigma_k$. 
Here $\nu $ denotes the unit normal to $\Sigma_k$ pointing to $\infty$. 
Since $ 4 \pi \c_{_{\Sigma_k}} =  \int_{\Sigma_k} \p_\nu u_k $ and 
$ 4 \pi \c_{_{(\Sigma_k, \tilde S)}} = \int_{\Sigma_k }  \p_\nu v_k $, one has
$$  \beta_k^{-1}  \c_{_{\Sigma_k}}  \ge \c_{_{(\Sigma_k, \tilde S )} }  \ge   \c_{_{\Sigma_k}}  .  $$
The claim follows by noting that $\beta_k$ has a uniform positive lower bound as $k \to \infty$.
\end{proof}

As an application of Propositions \ref{prop-pmt-c-w} and \ref{prop-c-w-localized}, we have 

\begin{thm} \label{thm-pmt-cw}
Let $N$ be a noncompact,  connected, orientable $3$-manifold.
Suppose $H_2 (N) = 0 $. Let $ M = N \setminus \{ p \}$ where $ p  $ is a point in $N$. 
Let $g$ be a smooth metric with nonnegative scalar curvature on $M$ such that  $g$ is  asymptotically flat near $p$. 
Assume there is a precompact domain $D \subset N$ such that $ p \in D$ and 
$( N \setminus D, g) $ is isometric to 
$$ ( (0, \delta] \times \Sigma, \bar g  + h ) ,  $$
where 
\begin{itemize}

\item  $\delta > 0$ is a constant, $ \Sigma$ is a closed, connected, orientable surface; 

\vspace{.1cm}

\item $ \bar g = d r^2 + a(r)^2 \sigma $, in which $ \sigma$ is a given metric on $\Sigma$ 
and $a(r) $ is a positive function on $(0, \delta]$; and 

\vspace{.1cm}

\item $ \lambda^{-1} \le | \bar g + h |_{\bar g} \le \lambda $ for some constant $ \lambda >0 $.
\end{itemize} 
Then 
\be \label{eq-pmt-c-w-a-r}
\lim_{r \to 0}  \left(  \int_r^\delta \frac{1}{a (x)^2}  \, d x \right)^{-1}  \left[  | a'(r ) |  +   a(r) | \bar \nabla h |_{\bar g}  \right] = 0 \ 
 \Longrightarrow \  \m \ge 0 . 
\ee
\end{thm}

\vspace{.1cm}

\begin{rem} \label{rem-pmt-cone-horn}
If $ a (r) = r^b$ for a constant $ b >0$, 
then \eqref{eq-pmt-c-w-a-r} translates to 
$$ \lim_{r \to 0} \,  r^{3b -2} \left( 1 + r | \bar \nabla h |_{\bar g} \right) = 0 \ \Rightarrow \ \m \ge 0 . $$ 
This in particular implies, if $g$ has a conical or $r^b$-horn type singularity modeled on $ \bar g = d r^2 + r^{2b} \sigma$ near $r=0$, 
then, under a mild asymptotic assumption of
$$  \lambda^{-1} \le | \bar g + h |_{\bar g} \le \lambda \ \ \text{and} \ \ r | \bar \nabla h |_{\bar g} = O (1), $$
one has  
``$ b > \frac23  \Rightarrow \m \ge 0 $". 
(Related results on PMT with isolated singularities can be found in \cite{ST18, LM17, DSW23}).
\end{rem}

\begin{proof}[Proof of Theorem \ref{thm-pmt-cw}] 
Let $ \Sigma_r = \{ r \} \times \Sigma$, $ r \in (0, \delta ]$. 
For  $ s \in (0, \delta)$, 
let $ \bar \c_{_{( \Sigma_s, \Sigma_\delta) } }$, $\bar W (\Sigma_s)$
denote the capacity of $\Sigma_s$ relative to $\Sigma_{\delta}$, the Willmore functional of $\Sigma_s$, respectively,
with respect to $\bar g$. 

The function 
$ u ( r) = \left(  \int_s^\delta a(x)^{-2} \, d x \right)^{-1} \,  \int_s^r a(x)^{-2} \, d x$ 
is $\bar g$-harmonic on $ [s , \delta ] \times \Sigma $ with $ u = 0 $ at $\Sigma_s$ and 
$u = 1$ at $ \Sigma_\delta$. 
This implies 
\be \label{eq-bar-c-s-d}
\bar \c_{_{( \Sigma_s, \Sigma_\delta) } }  = \frac{ |\Sigma|_{\sigma} }{4\pi}  \left( \int_s^\delta a(x)^{-2} \, d x \right)^{-1} , 
\ee
where $ | \Sigma|_\sigma$ is the area of $(\Sigma, \sigma)$. 
The mean curvature $\bar H$ of $\Sigma_s$ with respect to $\bar g$ is $ \bar H = 2 a^{-1} a' $. Hence,  
$$ \bar W (\Sigma_s) =  4 a'(s)^2  | \Sigma |_\sigma  . $$

We compare  $ \bar \c_{_{( \Sigma_s, \Sigma_\delta) } }$ and $  \c_{_{( \Sigma_s, \Sigma_\delta) } }$. 
Let $\bar \nabla$, $\nabla $ and $d V_{\bar g}$, $d V_g$ denote the gradient, the volume form 
with respect to $\bar g$, $g$, respectively.
Since $\c_{_{( \Sigma_s, \Sigma_\delta) } }$ equals the infimum of the $g$-Dirichlet energy of functions that 
vanish at $ \Sigma_s$ and equal $1$ at $\Sigma_\delta$, we have 
\be \label{eq-c-bar-c}
\begin{split}
 \c_{_{( \Sigma_s, \Sigma_\delta) } } \le  & \ \int_{[ s, \delta ] \times \Sigma} | \nabla u |_g^2 \, d V_g \\
 \le & \  C \int_{[ s, \delta ] \times \Sigma} | \bar \nabla u |_{\bar g}^2 \, d V_{\bar g} = C \, \bar  \c_{_{( \Sigma_s, \Sigma_\delta) } } .
\end{split}
\ee
Here $C > 0$ denotes a constant independent on $s$ and we have used the 
assumption $ \lambda^{-1} \le | g |_{\bar g} \le \lambda$. 

We also compare $ \bar W (\Sigma_s) $ and $ W (\Sigma_s) $. Let $\bar \Pi$ denote the second fundamental form 
of $\Sigma_s$ with respect to $\bar g$. Direct calculation shows 
\be
H - \bar H  =  | \bar \Pi |_{\bar g} \, O ( | h |_{\bar g} )  + O ( | \bar \nabla h |_{\bar g} ) . 
\ee  
(For instance, see formula (2.33) in \cite{MP21} and the proof therein.) 
Therefore, 
\be
\begin{split}
H^2 = & \ \bar H^2  + \left[  | \bar \Pi |_{\bar g} \, O ( | h |_{\bar g} )  + O ( | \bar \nabla h |_{\bar g} )   \right]^2 
 + \bar H  \left[  | \bar \Pi |_{\bar g} \, O ( | h |_{\bar g} )  + O ( | \bar \nabla h |_{\bar g} ) \right] \\
= & \  \bar H^2  + | \bar \Pi |_{\bar g}^2 \, O ( | h |_{\bar g} ) +  \bar H  \, O ( | \bar \nabla h |_{\bar g} )  +   O ( | \bar \nabla h |_{\bar g}^2 ) .
\end{split}
\ee
Let $ d \sigma_g$, $ d \sigma_{\bar g}$ denote the area form on $\Sigma_s$ with respect to $g$, $\bar g$, respectively. 
Then 
\be
\begin{split} 
  \int_{\Sigma_s}  H^2 \, d \sigma_{ g}  \le  & \ C \int_{\Sigma_s}  H^2  \, d \sigma_{ \bar g}  \\
 = & \  C \, \bar W (\Sigma_s)  
 +    \left[  | \bar \Pi |_{\bar g}^2 \, O ( | h |_{\bar g} ) 
  +  \bar H  \, O ( | \bar \nabla h |_{\bar g} ) +  O ( | \bar \nabla h |_{\bar g}^2 )   \right] \,  | \Sigma_s |_{\bar g} . 
\end{split}
\ee
Plugging in $ \bar W (\Sigma_s) = 4 {a'}^2 | \Sigma |_\sigma $, $ | \bar \Pi |^2_{\bar g} = 2 a^{-2} {a'}^2 $, and 
$ \bar H^2 = 4  a^{-2} {a'}^2  $, we have
\be \label{eq-w-bar-w}
\begin{split} 
 W (\Sigma_s) 
   \le & \  C \, | \Sigma|_{\sigma}  \left[ {a'}^2 + {a'}^2 \, | h |_{\bar g} 
   + a a' | \bar \nabla h |_{\bar g}  +  a^2 | \bar \nabla h |_{\bar g}^2  \right]   \\
 \le & \  C \, | \Sigma|_{\sigma} \left[ {a'}^2 ( 1 + | h |_{\bar g}) +  a^2 | \bar \nabla h |_{\bar g}^2  \right] .  
\end{split}
\ee 
As $ | h |_{\bar g} $ is bounded by assumption, it follows from \eqref{eq-bar-c-s-d}, \eqref{eq-c-bar-c} and \eqref{eq-w-bar-w} that
\be
\c_{_{( \Sigma_s, \Sigma_\delta) } } W (\Sigma_s)^\frac12 \le  C  \left( \int_s^\delta a(x)^{-2} \, d x \right)^{-1} 
   \left[ |a'|  +  a  | \bar \nabla h |_{\bar g}  \right] .
\ee
\eqref{eq-pmt-c-w-a-r} now follows from Propositions \ref{prop-pmt-c-w} and \ref{prop-c-w-localized}.
\end{proof}

\vspace{.1cm}

\begin{rem}
The negative mass Schwarzschild manifolds are known to have an $r^b$-horn type singularity with $b = \frac23$
(see \cite{BJ13, ST18} for instance). 
In \cite{BJ13}, Bray and Jauregui developed a theory of ``zero area singularities" (ZAS) modeled on the singularity of these manifolds.
Among other things, they introduced a notion of the mass of ZAS.  
In \cite[Theorem 4.8]{Robbins09}, Robbins showed the ADM mass of an asymptotically flat $3$-manifold with a single ZAS 
is at least the ZAS mass. 
The conclusion on the $r^b$-horn type singularity in Remark \ref{rem-pmt-cone-horn} 
 can also be reached via the results on ZAS in \cite{BJ13, Robbins09}. 
\end{rem}

We end this section by applying \eqref{eq-main-m-c-ratio} to obtain 
information of $ W( \cdot)$  in the negative mass Schwarzschild manifolds.

\begin{prop} \label{prop-singularity}
Consider a spatial Schwarzschild manifold  with {\sl negative} mass, i.e. 
$$ (M_{\m} , g_{\m} ) = \left( (0, \infty) \times S^2 , \frac{1}{ 1 + \frac{2m}{r} } dr^2 + r^2 \sigma_o \right) , $$
where $ (S^2, \sigma_o)$ denotes the standard unit sphere and the mass $ \m = - m  $ is negative. 
Let $ \Sigma  \subset M_{\m} $ be any connected, closed surface that is homologous to a slice $\{r \} \times S^2$. 
Let 
$ r_{max} (\Sigma) = \max_{_{ x \in \Sigma } } \, r (x). $
Then 
\be \label{eq-W-in-SC-n-m}
W(\Sigma) \ge  16 \pi \left( 1 + \frac{2m}{ r_{max} (\Sigma)  } \right) .
\ee
In particular, $ W (\Sigma ) \to  \infty $ as $ r_{max} (\Sigma)  \to 0$. 
\end{prop}

\begin{proof}
\eqref{eq-main-m-c-ratio} implies 
\be
  \left( \frac{1}{16 \pi} W(\Sigma) \right)^\frac12 \ge 1 + \frac{m}{ \c_{_\Sigma}} . 
\ee
Let $ \Sigma_* =   \{ r_{max} (\Sigma) \} \times \mathbb{S}^2$. 
As $ \Sigma_* $ encloses $ \Sigma$, 
\be \label{eq-compare-c}
 \c_{_{\Sigma_*}} \ge \c_{_\Sigma} . 
\ee 
Since $ \m = - m < 0 $, the above implies 
\be \label{eq-W-2}
  \left( \frac{1}{16 \pi} W(\Sigma) \right)^\frac12 \ge 1 + \frac{m}{ \c_{_{\Sigma_* } } } .
\ee
Direct calculation gives
\be \label{eq-c-symmetry}
 \c_{_{\Sigma_*} } =  \frac{ m }{\left( 1 + \frac{2m}{ r_{max} (\Sigma) } \right)^\frac12 - 1 }  . 
\ee
\eqref{eq-W-in-SC-n-m} follows from \eqref{eq-W-2} and \eqref{eq-c-symmetry}. 
\end{proof}

Proposition \ref{prop-singularity} gives another perspective of 
the singularity of $(M_{\m}, g_{\m}) $ via the Willmore functional $W(\cdot)$.

\section{Bounding $\m$ via $ \c (M,g)$}  \label{sec-pmt-cm-positive}

Let $(M, g)$ be given in the definition of $\c (M,g)$ in \eqref{eq-inf-c}. 
In this section, we relate $ \m $ and $ 2 \c (M,g)$ assuming $\c (M,g) > 0$.
We begin with a characterization of $ \c (M, g) > 0 $ which follows from standard arguments on harmonic functions. 

\begin{prop} \label{prop-c-m-positive}
Let $\c (M,g)$ be defined in \eqref{eq-inf-c}. Then 
$\c (M, g) > 0 $ if and only if there exists 
a harmonic function $ w $ on $(M,g)$ such that $ 0 < w < 1 $ on $M$ and $ w (x) \to 1$ at $\infty$ (i.e. as $x \to p$).  
\end{prop}

\begin{proof}
For each $k$, let $u_k$ be the harmonic function on $(E_{_{\Sigma_k} }, g)$ with $ u_k \to 1 $ as $x \to \infty$
and $ u_k = 0 $ at $\Sigma_k$. 
Given any surface $ S \in \mathcal{S}$, by the maximum principle, 
$\{ u_k \}$ forms an increasing sequence in the exterior of $S$ relative to $\infty$ (i.e. in $ D_{_S} \setminus \{ p \}$). 
Interior elliptic estimates imply $ \{ u_k \} $ converges to a harmonic function $u_\infty$ on $M$
uniformly on compact sets in any $C^i$-norm. 
The limit $u_\infty$ satisfies $ 0 < u_\infty \le 1$ and $ u_\infty \to 1$ as $ x \to \infty$. 
By the strong maximum principle, either $ u_\infty \equiv 1$ or $ 0 < u_\infty < 1 $.   

Suppose $(M,g)$ admits a harmonic $w $ with $ 0 < w < 1 $ and $ w \to 1 $ at $\infty$. Then 
$w$ is an upper barrier for $\{ u_k \}$, which implies $ u_\infty \le w $, and hence $ 0 < u_\infty < 1 $. 
In this case, $ \c (M, g) $ must be positive. Otherwise, if $\c (M,g) = 0 $, then 
 $\lim_{k \to \infty} \int_{E_{_{\Sigma_k}}  } | \nabla u_k |^2 = 0 $, which would imply  $ \int_{K} | \nabla u_\infty |^2 = 0 $ 
 on any compact set $K $ in $M$, and hence $u_\infty \equiv 1$, a contradiction.  

Next suppose $ \c (M, g) > 0 $. We want to show $ 0 < u_\infty < 1 $. If not, $ u_\infty \equiv 1$ on $M$. 
Pick any surface $ S \in \mathcal{S}$, then $ \lim_{k \to \infty} u_k  = 1 $ at $S$.  
Let $ \beta_k = \min_{S} u_k $ for large $k$. Let $ \tilde u_k $ be the harmonic function on $ E_{_S} $
with $ \tilde u_k \to 1$ at $\infty$ and $ \tilde u_k = \beta_k $ at $S$.
By the maximum principle, $ u_k \ge \tilde u_k $. Therefore, 
$ \tilde c_k \ge c_{k}$, where $\tilde c_k$, $c_k$ are the coefficients in the expansions
$$ \tilde u_k = 1 - \tilde c_k |x|^{-1} + o ( | x|^{-1} ) , $$
$$ u_k = 1 - c_k |x|^{-1} + o ( |x|^{-1} ) ,  $$
as $x \to \infty$. Here we have $c_k = \c_{_{S_k}} $ and 
$$ 4 \pi \tilde c_k = \lim_{r \to \infty} \int_{|x| = r} \frac{\p \tilde u_k}{\p  \nu} = \int_{S } \frac{\p \tilde u_k}{\p \nu} , $$
where $ \nu$ denotes the corresponding unit normal pointing to $\infty$.  
Elliptic boundary estimates applied to $w_k = \tilde u_k - \beta_k$ shows
$$ \lim_{k \to \infty}  \max_{S} | \nabla  w_k | = 0 . $$
 Consequently, $ \tilde c_k \to 0$ as $ k \to \infty$.  
 Combined with $ \tilde c_k \ge \c_k $, this shows 
 $$ \c (M,g) = \lim_{ k \to \infty}  \c_{_{S_k}}  = 0 , $$
which is a contradiction. Therefore, $ 0 < u_\infty < 1 $.
This completes the proof.
\end{proof}

\begin{rem}
One may further require $w $ satisfies $ \inf_M w = 0 $ in Proposition \ref{prop-c-m-positive}. 
To see this, it suffices to examine the proof beginning with assuming  $ \c (M,g) > 0 $.
In this case, we have shown $  0 < u_\infty < 1 $ on $M$. 
Suppose $ \inf_M u_\infty > 0 $, consider 
$ v  =  ( 1 - \inf_M u_\infty )^{-1}   ( u_\infty - \inf_M u_\infty ) $. 
Then $ v < u_\infty$ and $ v $ also acts as a barrier for $ \{ u_k \}$, which implies $u_\infty \le v $, a contradiction.
Hence, $ \inf_M u_\infty = 0 $. 
\end{rem}

Next, we focus on the case in which the function $u$ tends to zero at ``the other end". 

\begin{prop} \label{prop-c-m-c-u}
Suppose there is a harmonic function $ u $ on $(M,g)$ with $ 0 < u < 1 $, $ u (x) \to 1$ at $\infty$ (i.e. as $x \to p$),
and $ \lim_{k \to \infty} \max_{\Sigma_k} u = 0 $. 
Then 
\be \label{eq-c-m-c-u}
\c (M, g) = \mathcal{C},
\ee
where $\mathcal{C} > 0 $ is the coefficient in the expansion of 
$$ u = 1 - \C |x|^{-1} + o ( | x|^{-1} )  $$
in the AF end.
\end{prop}

\begin{proof}
Let $u_k$ be the harmonic function on $(E_{_{\Sigma_k}}, g)$ with $ u_k \to 1 $ at $\infty$
and $ u_k = 0 $ at $\Sigma_k$. Then $ u_k \le u $,  which implies $ \mathcal{C} \le c_k $, 
where $c_k = \c_{_{\Sigma_k}} $ is the coeffiicent in 
$$ u_k = 1 - c_k |x|^{-1} + o ( |x|^{-1} )  $$ 
as $x \to \infty$. This shows 
$$ \C \le \lim_{k \to \infty} \c_{_{\Sigma_k} }  = \c (M,g) . $$

To show the other direction, consider $ \alpha_k = \max_{\Sigma_k} u $. 
On $E_{_{\Sigma_k}}$, by the maximum principle, 
$ \frac{ 1}{1 - \alpha_k} ( u - \alpha_k ) \le u_k $, which implies 
$ \frac{1}{1 - \alpha_k} C \ge c_k $. As $ \alpha_k \to 0$, this gives 
$$ \C \ge \lim_{ k \to \infty } ( 1  - \alpha_k ) \c_{_{\Sigma_k} }   = \c (M, g) . $$ 
Therefore, $ \C   = \c (M, g)$.
\end{proof}

We are now in a position to derive applications of \eqref{eq-gdu-mc}.

\begin{thm} \label{thm-c-m-u-limit}
Let $N$ be a noncompact,  connected, orientable $3$-manifold
admitting an exhaustion sequence of precompact domains $D_k$ with
connected boundary $\p D_k$, $ k = 1 , 2, \ldots$. 
Suppose $H_2 (N) = 0 $. Let $ M = N \setminus \{ p \}$ where $ p  $ is a point in $N$. 
Let $g$ be a smooth metric with nonnegative scalar curvature on $M$ such that  $g$ is  asymptotically flat near $p$. 
Assume there is a harmonic function $ u $ on $(M,g)$ with $ 0 < u < 1 $, $ u (x) \to 1$ as $x \to p$,
and $ \displaystyle \lim_{k \to \infty} \max_{{\p D_k}} u = 0 $. 
Then
\begin{enumerate}
\item[(i)] The limit $ \displaystyle \lim_{t \to 0} \int_{u^{-1} (t) } | \nabla u |^2 $ exists (finite or $\infty$),
where $t \in (0,1)$ is a regular value of $u$; and

\item[(ii)] 
$ \displaystyle  \m \ge 2  \c (M,g)  \left[ 1 -  \lim_{t \to 0}  \left( \frac{1}{4 \pi} \int_{u^{-1} (t)} | \nabla u |^2 \right)^\frac12  \right] $. 

\end{enumerate}
\end{thm}

\begin{proof}
Given a regular value $ t \in (0,1)$ of $u$, let $ \Sigma_t = u^{-1} (t)$. $\Sigma_t$ is a closed, orientable surface in 
$M = N \setminus \{ p \}$. Let $ \Sigma_t^{(1)} $ denote any connected component of $\Sigma_t$.
Since $H_2 (N) = 0 $, $\Sigma_t^{(1)}$ is the boundary of a bounded domain $\Omega_1$ in $N$. If $ p \notin \Omega_1$, then $u$ is 
identically a constant by the maximum principle. Hence, $ \Sigma_t^{(1)}$ encloses $p$. 
As a result, if there are two connected components of $\Sigma_t$, then 
both of them enclose $p$, and thus form the boundary of a bounded domain in $M$. 
By the maximum principle, $u$ is a constant, which is a contradiction. 
Therefore, $\Sigma_t$ is connected. 
Since $t$ is arbitrary, this in particular shows \eqref{eq-gdu-mc} is applicable to $(E_t, g)$, 
where $E_t = \{ u \ge t \} \subset M $ is the exterior of $\Sigma_t$ with respect to $\infty$. 

Applying \eqref{eq-gdu-mc} to $(E_t, g)$, we have
\be  \label{eq-gdu-mc-E-t}
\frac{\m}{ 2  \c_{_{\Sigma_t}} } \ge 1 - 
\left( \frac{1}{4 \pi} \int_{\Sigma_t} | \nabla u_t |^2 \right)^\frac12 .
\ee
Here $ u_t  = \frac{1 }{1-t} ( u - t) $ is the harmonic function on $(E_t, g)$ that tends to $1$ at $\infty$ and equals $0$ at $ \Sigma_t$,
$ \c_{_{\Sigma_t} } = \frac{1}{1-t} \C $, and   
$\C$ is the coefficient in the expansion of 
$$ u = 1 - \C |x|^{-1} + o ( | x|^{-1} )  . $$
It follows from \eqref{eq-gdu-mc-E-t} that  
\be  \label{eq-gdu-mc-E-t-1}
\frac{\m}{ 2  \C } \ge \frac{1}{1-t} - 
\frac{1}{(1-t)^2}  \left( \frac{1}{4 \pi} \int_{\Sigma_t} | \nabla u |^2 \right)^\frac12 .
\ee

Consider the function
$$ \mathcal{B} (t) = \frac{1}{(1-t) }  \left[ 4 \pi - \frac{1}{(1-t)^2} \int_{ \Sigma_t } | \nabla u |^2 \right]  . $$
In \cite[Theorem 3.2 (ii)]{M22}, we showed 
$ \mathcal{B}(t)$ is monotone nondecreasing in $ t$ if $ g$ has nonnegative scalar curvature. 
As a result, 
\bee
\lim_{t \to 0 } \mathcal{B} (t) \  \text{exists}.
\eee
Consequently, 
\bee
\lim_{t \to 0 } \int_{\Sigma_t} | \nabla u |^2  \  \text{exists}. 
\eee
This proves (i). (ii) follows from \eqref{eq-gdu-mc-E-t-1}, (i) and Proposition \ref{prop-c-m-c-u}. 
\end{proof}

We have not assumed $g$ to be complete on $M$ so far. 
In particular, $(M,g)$ in Theorem \ref{thm-c-m-u-limit}
could just be the interior of an AF manifold with boundary $\Sigma$ and the function $u$ may simply be the restriction,  
to the interior, of the harmonic function that tends to $1$ at $\infty$ and vanishes at $\Sigma$. 
In that extreme case, $ \lim_{t \to 0 } \int_{\Sigma_t} | \nabla u |^2  = \int_{\Sigma} | \nabla u |^2$ and 
(ii) reduces to \eqref{eq-gdu-mc}. 

\vspace{.2cm}

If $g$ is complete on $M$, we have the following corollary.

\begin{cor} \label{thm-m-c-Ric}
Let $N$, $p$, $M$, $g$ and $u$ be given as in Theorem \ref{thm-c-m-u-limit}. 
Suppose $(M,g)$ is complete and has Ricci curvature bounded from below.
Then
\be \label{eq-m-2c}
\m \ge 2 \C,
\ee
where $ \C =  \c (M,g)$ is the coefficient in the expansion of 
$$ u = 1 - \C |x|^{-1} + o ( | x|^{-1} )  $$
as $x \to p$. 
\end{cor}

Corollary \ref{thm-m-c-Ric} relates to a result of Bray \cite{Bray02}.
In \cite[Theorem 8]{Bray02}, Bray proved that, if $(M,g)$ is a complete asymptotically flat $3$-manifold with nonnegative 
scalar curvature which has multiple AF ends and mass $\m$ in a chosen end, then 
$$ \m \ge 2 \C , $$
where $\C$ is the coefficient in the expansion $  u = 1 - \C |x|^{-1} + o ( | x|^{-1} )  $ at the chose end, and 
$u$ is the harmonic function that tends to $1$ at the chosen end and approaches $0$ at all other AF ends. 

Bray's theorem allows $M$ to have more general topology and more than two ends. Its proof made use of the $3$-d PMT. 
Complete manifolds whose ends are all asymptotically flat necessarily have bounded Ricci curvature. 
In this sense, Corollary \ref{thm-m-c-Ric} provides a partial generalization of Bray's result. 

\begin{proof}[Proof of Corollary \ref{thm-m-c-Ric}]
Let $\Sigma_t $ be given in the proof of Theorem \ref{thm-c-m-u-limit}. 
Since $(M, g)$ is complete and has Ricci curvature bounded from below, 
by the gradient estimate of Cheng and Yau \cite{ChengYau75}, 
$ \max_{\Sigma_t}  \, u^{-1} | \nabla u | \le  \Lambda $ 
where $\Lambda$ is a constant independent on $t$. 
Combined with $ \int_{\Sigma_t} | \nabla u | = 4 \pi \C $, this shows
\bee
\frac{1}{4\pi} \int_{\Sigma_t} | \nabla u |^2   \le \C \, \Lambda t  ,
\eee
which implies 
\be \label{eq-app-CY-1}
\lim_{t \to 0} \int_{\Sigma_t} | \nabla u |^2  = 0 .
\ee
It follows from \eqref{eq-app-CY-1} and Theorem \ref{thm-c-m-u-limit} (ii) that $  \m \ge 2 \C $.  
\end{proof}

\begin{rem}
Let $ R$, $ \Ric$ denote the scalar curvature, Ricci curvature of $g$. Since   
$$ R \ge 0 \ \text{and} \ \Ric \ \text{bounded from above}  \ \Rightarrow  \  \Ric \  \text{bounded from below} , $$
Corollary \ref{thm-m-c-Ric} also holds if the assumption of ``$ \Ric$ bounded from below" is replaced by 
``$ \Ric$ bounded from above".  
\end{rem}

\begin{rem} \label{rem-m-c}
As used in Bray's work \cite{Bray02}, the inequality $ \m \ge 2 \C$ has a geometric interpretation that asserts 
 the mass of the conformally deformed metric $ u^4 g $, which might not be complete, is nonnegative. 
Instead of $\m \ge 2 \C$, a weaker inequality $\m \ge \C$ was obtained by Hirsch, Tam and the author in 
\cite{HMT22}. 
\end{rem}
 
\vspace{.4cm}

\noindent {\sl Acknowledgement.}
I thank Christos Mantoulidis who introduced me to the work \cite{MT22}. 
I also thank Professor Xianzhe Dai for several helpful communications.

\end{document}